\documentclass[a4paper,12pt]{article}
 \usepackage{graphicx}
\usepackage{tabularray}
\UseTblrLibrary{counter,varwidth}
 \usepackage[top=2.5cm,bottom=2.5cm,left=2.5cm,right=2.5cm]{geometry}
 \usepackage{cite, amsmath, amssymb, amsthm, amsfonts, color, verbatim, emptypage}
 \usepackage{enumerate}
 \usepackage[margin=1cm,%
 font=small,%
 format=hang,%
 labelsep=period,%
 labelfont=bf]{caption}
\usepackage[dvipsnames]{xcolor}

\usepackage{pgfplots}
\usepackage{hyperref}
\pgfplotsset{compat=1.15}
\usepackage{mathrsfs}
\usetikzlibrary{arrows}
 \usepackage[para]{footmisc}

 \usepackage{pgf,tikz,pgfplots}
 \usepackage{mathrsfs}
 \usetikzlibrary{arrows}
 \usetikzlibrary{decorations.pathreplacing}
 \tikzstyle{vertex}=[circle,fill=black,inner sep=0pt]
\tikzstyle{edge} = [draw,thick]
\tikzstyle{weight} = [font=\small,fill=white, draw=white,inner sep=0pt, opacity=0.0, text opacity=1]

 \date{ }
  
 \normalsize

 \newtheorem{theorem}{Theorem}
 \newtheorem{corollary}[theorem]{Corollary}

\newtheoremstyle{mytheoremstyle1} 
        {\topsep}                    
        {\topsep}                    
        {}                   
        {}                           
        {\fontfamily{thm}\selectfont\scshape\color{black}}                   
        {.}                          
        {.5em}                       
        {}  
\theoremstyle{mytheoremstyle1}
 
 \newtheorem{case}{Case}

 \tikzstyle{every node}=[circle, draw, fill=black,inner sep=0pt, minimum width=4pt]
 
 \numberwithin{subcase}{case}
 \theoremstyle{definition}
 \newtheorem{remark}{Remark}
 \newtheorem{lemma}{Lemma}

 \usepackage{authblk}

\title{\bf \vskip 3cm $2$-Coupon Coloring of Cubic Graphs Containing $3$-Cycle or $4$-Cycle }
\author[]{S. Akbari\footnote{s\_akbari@sharif.edu}}
\author[]{M. Azimian\footnote{mahnooshmh22@gmail.com}}
\author[]{A. Fazli Khani\footnote{ar.fazlikhani@gmail.com}}
\author[]{B. Samimi\footnote{behrad.s138282@gmail.com}}
\author[]{E. Zahiri\footnote{ez1382@gmail.com}}
\affil[]{\small Department of Mathematical Sciences, Sharif University of Technology, Tehran, Iran}

\date{}							

\begin{document}
\maketitle
\begin{sloppypar}
\begin{abstract}
Let $G$ be a graph. A total dominating set in a graph $G$ is a set $S$ of vertices of $G$ such that every vertex in $G$ is adjacent to a vertex in $S$. Recently, the following question was proposed:\\
``Is it true that every connected cubic graph containing a $3$-cycle has two vertex disjoint total dominating sets?"\\
In this paper, we give a negative answer to this question. Moreover, we prove that if we replace $3$-cycle with $4$-cycle the answer is affirmative. This implies every connected cubic graph containing a diamond (the complete graph of order $4$ minus one edge) as a subgraph can be partitioned into two total dominating sets, a result that was proved in $2017$. 
\end{abstract}

\baselineskip=0.30in

\section {Introduction}
Throughout this paper, all graphs are simple that is with no loop and multiple edges. Let $G$ be a graph. We denote the vertex set and the edge set of $G$ by $V(G)$ and $E(G)$, respectively. For $v \in V(G)$, $d(v)$ denotes the degree of $v$. A vertex $v$ is called an {\it $i$-vertex} if $d(v)=i$.  A {\it total dominating set} in a graph $G$ is a set $S \subseteq V(G)$ such that every vertex in $G$ is adjacent to a vertex in $S$.
The study of cubic graphs whose vertex set can be partitioned into two total dominating sets is an attractive topic and many authors have investigated this problem.  A {\it $k$-coloring} of a graph $G$ is an assignment of colors from the set $[k] = \left\{0, \ldots , k-1 \right\}$ to the vertices of $G$.
A {\it $k$-coupon} coloring of a graph $G$ without isolated vertices is a $k$-coloring of $G$ such that the neighborhood of every vertex of $G$ contains vertices of all colors from $[k]$. The maximum $k$ for which a $k$-coupon coloring exists is called the {\it coupon coloring number} of $G$. We say that a vertex has {\it coupon property} if all $k$ colors are appeared in the neighbors of that vertex.
The concept of $k$-coupon coloring of graphs has been studied by several authors, for instance, see [\ref{1}], [\ref{4}] and [\ref{5}].\\
The {\it open neighborhood} of a vertex $v$ in  $G$ denoted by $N_G(v)$ is the set of all vertices adjacent to $v$. In this paper for any $S \subseteq V(G)$, $G[S]$ denotes the subgraph of $G$ induced by $S$.
A {\it $\left\{ 1,2 \right\} $-factor} is a spanning subgraph of $G$ in which each component is either $1$-regular or $2$-regular.
We use $C_n$ and $P_n$, for the cycle and the path of order $n$, respectively. Also, $C_n$ is called an {\it $n$-cycle}. A vertex $v$ is called a {\it good vertex} or a {\it bad vertex}, if it has or has not the coupon property, respectively.
It is not hard to see that a $2$-coupon coloring of a graph is equivalent to the vertex partitioning of a graph into two total dominating sets. In [\ref{2}] the following interesting question was proposed:\\
``{\it Is it true that every connected cubic graph containing a $3$-cycle has two vertex disjoint total dominating sets?}"\\
Here, we construct a connected cubic graph of order $60$ containing a $3$-cycle whose vertex set cannot be partitioned into two total dominating sets.
The complete graph of order $4$ minus one edge is called a {\it diamond}. In [\ref{2}] it is shown that the vertex set of a cubic graph containing a diamond as a subgraph can be partitioned into two total dominating sets.
In this paper, we strengthen this result by showing that the vertex set of every cubic graph containing a $4$-cycle can be partitioned into two total dominating sets.

For a graph $G$, we use the symbol $G^*$ for a bipartite graph that is constructed from $G$ in the following way:\\
Let $V(G)=\left\{ v_1, \ldots , v_n \right\}$.
Form a bipartite cubic graph $G^*=(U,W)$ of order $2n$, where $U= \left\{ u_1, \ldots , u_n \right\}$ and $W= \left\{ w_1, \ldots , w_n \right\}$. Join $u_i$ to $w_j$ and $u_j$ to $w_i$ if and only if $v_iv_j \in E(G)$.
\newpage
\section{Main Results}
In this paper, we denote the Heawood graph by $\mathcal{H}$.
\begin{figure}[htp]
\centering
\begin{tikzpicture}[line cap=round,line join=round,>=triangle 45,x=1cm,y=1cm, scale=0.6]
\draw [line width=2pt] (0.21,1.8) circle (4cm);
\draw [line width=2pt] (2.70917496591696,4.923159376294134)-- (1.093574220842098,-2.1011916892487195);
\draw [line width=2pt] (3.8167669077788298,3.5295179886175045)-- (-3.3909739976776097,3.5415470903911093);
\draw [line width=2pt] (4.209994429461842,1.793324352149031)-- (-2.2891749659169593,-1.3231593762941343);
\draw [line width=2pt] (3.8109739976776082,0.058452909608889714)-- (-0.6735742208421005,5.701191689248719);
\draw [line width=2pt] (2.698736502605444,-1.3314837730058924)-- (-3.7899944294618426,1.8066756478509678);
\draw [line width=2pt] (-0.68659077168571,-2.0982207464595977)-- (-2.2787365026054465,4.931483773005891);
\draw [line width=2pt] (-3.3967669077788294,0.07048201138249421)-- (1.1065907716857075,5.698220746459598);
\begin{scriptsize}
\draw [fill=black] (1.1065907716857075,5.698220746459598) circle (5pt);
\draw [fill=black] (-0.6735742208421005,5.701191689248719) circle (5pt);
\draw [fill=black] (-2.2787365026054465,4.931483773005891) circle (5pt);
\draw [fill=black] (-3.3909739976776097,3.5415470903911093) circle (5pt);
\draw [fill=black] (-3.7899944294618426,1.8066756478509678) circle (5pt);
\draw [fill=black] (-3.3967669077788294,0.07048201138249421) circle (5pt);
\draw [fill=black] (-2.2891749659169593,-1.3231593762941343) circle (5pt);
\draw [fill=black] (-0.68659077168571,-2.0982207464595977) circle (5pt);
\draw [fill=black] (1.093574220842098,-2.1011916892487195) circle (5pt);
\draw [fill=black] (2.698736502605444,-1.3314837730058924) circle (5pt);
\draw [fill=black] (3.8109739976776082,0.058452909608889714) circle (5pt);
\draw [fill=black] (4.209994429461842,1.793324352149031) circle (5pt);
\draw [fill=black] (3.8167669077788298,3.5295179886175045) circle (5pt);
\draw [fill=black] (2.70917496591696,4.923159376294134) circle (5pt);
\end{scriptsize}
\end{tikzpicture}
\caption*{The Heawood graph}
\end{figure}
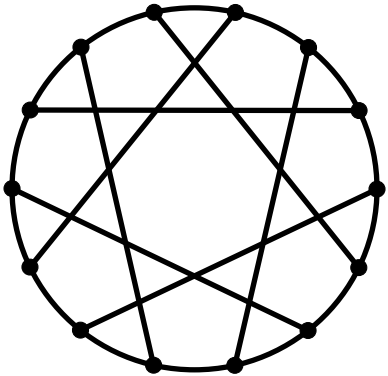

Note that $\mathcal{H} = (X, Y)$ is a vertex transitive bipartite cubic graph.

We start this section with the following lemma.
\begin{lemma}\label{lemma1}
In any $2$-coloring of $\mathcal{H}$, there are at least two bad vertices.
\end{lemma}
\begin{proof}
 By contradiction and using [\ref{2}, Observation $7$], there exists a $2$-coloring of $\mathcal{H} = (X, Y)$ with exactly one bad vertex.
With no loss of generality assume that $X$ has no bad vertex. Since $\mathcal{H}$ is vertex transitive, one can color the part $Y$ such as $X$ to obtain a $2$-coupon coloring for $\mathcal{H}$, a contradiction. The proof is complete.
\end{proof}

\begin{lemma}\label{lemma2}
 Let $e=uv \in E(\mathcal{H})$. If there is a $2$-coloring for $\mathcal{H} \backslash e$ such that all vertices in $V(\mathcal{H})\backslash \left\{ u, v \right\}$ are good, then $u$ and $N_{\mathcal{H} \backslash e}(v)$ have the same color.
\end{lemma}

\begin{proof}
By contradiction assume that $w \in N_{\mathcal{H} \backslash e}(v)$  and the color of $u$ and $w$ are different. It is straightforward  to see that it is a $2$-coloring of $\mathcal{H}$ in which every vertex in $V(\mathcal{H})\backslash {u}$ is good which contradicts Lemma \ref{lemma1}. The proof is complete.
\end{proof}

\begin{theorem}\label{theorem1}
There exists a connected cubic graph containing a $3$-cycle with no  $2$-coupon coloring.
\end{theorem}
\begin{proof}
Consider $4$ vertex disjoint copies of $\mathcal{H}$ and denote them by $\mathcal{H}_1, \mathcal{H}_2, \mathcal{H}_3, \mathcal{H}_4$.
Let $e_i = u_iv_i,$ for $i = 1, \ldots , 4$.
Now, we would like to form a cubic graph containing a triangle. Let $K$ be the following graph:
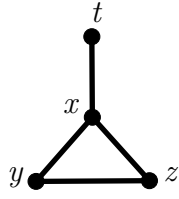
\begin{center}
\begin{tikzpicture}[line cap=round,line join=round,>=triangle 45,x=1cm,y=1cm]
\draw [line width=2pt] (-0.47355268078765284,-1.960472469103462)-- (0.2772256583516088,-2.801665124109507);
\draw [line width=2pt] (-1.2137909879434032,-2.810954947450909)-- (0.2772256583516088,-2.801665124109507);
\draw [line width=2pt] (-0.47355268078765284,-1.960472469103462)-- (-1.2137909879434032,-2.810954947450909);
\draw [line width=2pt] (-0.48020914610687515,-0.8921097853682758)-- (-0.47355268078765284,-1.960472469103462);
\draw (-1.046614282666777,-1.5350283416247968) node[anchor=north west,draw=none,fill=none,rectangle,inner sep=5pt] {$\mathit{x}$};
\draw (0.28069166093878595,-2.4365584574546766) node[anchor=north west,draw=none,fill=none,rectangle,inner sep=5pt] {$\mathit{z}$};
\draw (-1.7752780243303032,-2.4365584574546766) node[anchor=north west,draw=none,fill=none,rectangle,inner sep=5pt] {$\mathit{y}$};
\draw (-0.6794375773901508,-0.25) node[anchor=north west,draw=none,fill=none,rectangle,inner sep=5pt] {$\mathit{t}$};
\begin{scriptsize}
\draw [fill=black] (-0.48020914610687515,-0.8921097853682758) circle (3pt);
\draw [fill=black] (-0.47355268078765284,-1.960472469103462) circle (3pt);
\draw [fill=black] (-1.2137909879434032,-2.810954947450909) circle (3pt);
\draw [fill=black] (0.2772256583516088,-2.801665124109507) circle (3pt);
\end{scriptsize}
\end{tikzpicture}
\end{center}
Consider $(\bigcup_{i=1}^{4}  \mathcal{H}_i\backslash e_i) \cup  K$.
Join $t$, $y$, and $z$ to $\left\{ v_1, v_2 \right\}$, $u_3$, and $u_4$, respectively. Also join $u_1$ and $v_4$ to $u_2$ and $v_3$, respectively and call the resultant cubic graph by $G$.
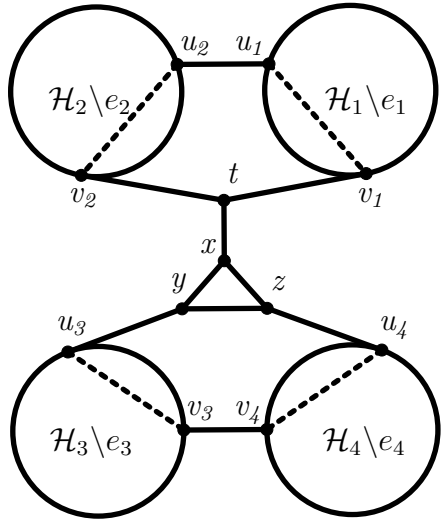
\begin{center}
\begin{tikzpicture}[line cap=round,line join=round,>=triangle 45,x=1cm,y=1cm,scale=0.75]
\draw [line width=2pt] (1.7351947457366965,1.0544304209476216) circle (1.5cm);
\draw [line width=2pt] (-2.719696054048024,1.026674091977) circle (1.5cm);
\draw [line width=2pt] (-0.47355268078765284,-1.960472469103462)-- (0.2772256583516088,-2.801665124109507);
\draw [line width=2pt] (-1.2137909879434032,-2.810954947450909)-- (0.2772256583516088,-2.801665124109507);
\draw [line width=2pt] (-0.47355268078765284,-1.960472469103462)-- (-1.2137909879434032,-2.810954947450909);
\draw [line width=2pt] (-0.48020914610687515,-0.8921097853682758)-- (-0.47355268078765284,-1.960472469103462);
\draw [line width=2pt] (-1.3000367311971461,1.5109956866252099)-- (0.3096107137495623,1.5210246426684917);
\draw [line width=2pt] (-0.48020914610687515,-0.8921097853682758)-- (2.014489312300446,-0.419338402541316);
\draw [line width=2pt] (-0.48020914610687515,-0.8921097853682758)-- (-2.980604929642854,-0.4504604850456355);
\draw [line width=2pt] (-2.6823001073120025,-4.975375359154546) circle (1.5cm);
\draw [line width=2pt] (1.772590692472719,-4.947619030183925) circle (1.5cm);
\draw [line width=2pt] (0.2772256583516088,-2.801665124109507)-- (2.288971964400125,-3.5393042990237937);
\draw [line width=2pt] (-1.2137909879434032,-2.810954947450909)-- (-3.2161897014737155,-3.573604371708364);
\draw [line width=2pt] (0.27259419650785754,-4.944376794360377)-- (-1.1824604651757114,-4.9534425554923605);
\draw [line width=2pt,dash pattern=on 2pt off 3pt] (0.3096107137495623,1.5210246426684917)-- (2.014489312300446,-0.419338402541316);
\draw [line width=2pt,dash pattern=on 2pt off 3pt] (-1.3000367311971461,1.5109956866252099)-- (-2.980604929642854,-0.4504604850456355);
\draw [line width=2pt,dash pattern=on 2pt off 3pt] (-3.2161897014737155,-3.573604371708364)-- (-1.1824604651757114,-4.9534425554923605);
\draw [line width=2pt,dash pattern=on 2pt off 3pt] (2.288971964400125,-3.5393042990237937)-- (0.27259419650785754,-4.944376794360377);
\draw (-0.55,2.25) node[anchor=north west,draw=none,fill=none,rectangle,inner sep=5pt] {$\mathit{u_1}$};
\draw (-1.5081013190536769,2.25) node[anchor=north west,draw=none,fill=none,rectangle,inner sep=5pt] {$\mathit{u_2}$};
\draw (1.6,-0.442032729129506244) node[anchor=north west,draw=none,fill=none,rectangle,inner sep=5pt] {$\mathit{v_1}$};
\draw (-3.4483597957115504,-0.442032729129506244) node[anchor=north west,draw=none,fill=none,rectangle,inner sep=5pt] {$\mathit{v_2}$};
\draw (1.992329475376112,-2.7) node[anchor=north west,draw=none,fill=none,rectangle,inner sep=5pt] {$\mathit{u_4}$};
\draw (-0.55,-4.2) node[anchor=north west,draw=none,fill=none,rectangle,inner sep=5pt] {$\mathit{v_4}$};
\draw (-3.669373784891187,-2.7) node[anchor=north west,draw=none,fill=none,rectangle,inner sep=5pt] {$\mathit{u_3}$};
\draw (-1.4,-4.2) node[anchor=north west,draw=none,fill=none,rectangle,inner sep=5pt] {$\mathit{v_3}$};
\draw (-1.146614282666777,-1.3350283416247968) node[anchor=north west,draw=none,fill=none,rectangle,inner sep=5pt] {$\mathit{x}$};
\draw (0.10069166093878595,-1.9865584574546766) node[anchor=north west,draw=none,fill=none,rectangle,inner sep=5pt] {$\mathit{z}$};
\draw (-1.6752780243303032,-1.9365584574546766) node[anchor=north west,draw=none,fill=none,rectangle,inner sep=5pt] {$\mathit{y}$};
\draw (1.0450666112135278,1.3796266250293113) node[anchor=north west,draw=none,fill=none,rectangle,inner sep=5pt] {$\mathcal{H}_1 \backslash e_1$};
\draw (1.0064890284359595,-4.643993639389179) node[anchor=north west,draw=none,fill=none,rectangle,inner sep=5pt] {$\mathcal{H}_4 \backslash e_4$};
\draw (-3.8,1.3796266250293113) node[anchor=north west,draw=none,fill=none,rectangle,inner sep=5pt] {$\mathcal{H}_2 \backslash e_2$};
\draw (-3.7754934215451975,-4.643993639389179) node[anchor=north west,draw=none,fill=none,rectangle,inner sep=5pt] {$\mathcal{H}_3 \backslash e_3$};
\draw (-0.6794375773901508,0) node[anchor=north west,draw=none,fill=none,rectangle,inner sep=5pt] {$\mathit{t}$};
\begin{scriptsize}
\draw [fill=black] (-0.48020914610687515,-0.8921097853682758) circle (3pt);
\draw [fill=black] (-0.47355268078765284,-1.960472469103462) circle (3pt);
\draw [fill=black] (-1.2137909879434032,-2.810954947450909) circle (3pt);
\draw [fill=black] (0.2772256583516088,-2.801665124109507) circle (3pt);
\draw [fill=black] (-1.3000367311971461,1.5109956866252099) circle (3pt);
\draw [fill=black] (0.3096107137495623,1.5210246426684917) circle (3pt);
\draw [fill=black] (2.014489312300446,-0.419338402541316) circle (3pt);
\draw [fill=black] (-2.980604929642854,-0.4504604850456355) circle (3pt);
\draw [fill=black] (-3.2161897014737155,-3.573604371708364) circle (3pt);
\draw [fill=black] (-1.1824604651757114,-4.9534425554923605) circle (3pt);
\draw [fill=black] (0.27259419650785754,-4.944376794360377) circle (3pt);
\draw [fill=black] (2.288971964400125,-3.5393042990237937) circle (3pt);
\end{scriptsize}
\end{tikzpicture}
\end{center}
We claim that $G$ has no $2$-coupon coloring.\\
By contradiction assume that $c$ is a $2$-coupon coloring for $G$.
With no loss of generality suppose that $c(v_1) = 0$, where $c(v_1)$ denotes the color of $v_1$.
By Lemma \ref{lemma2}, every vertex in $N_{\mathcal{H}_1 \backslash e_1}(u_1)$ has color $0$. Since $u_1$ is a good vertex, $c(u_2) = 1$. 
Thus by Lemma \ref{lemma2}, every vertex in $N_{\mathcal{H}_2 \backslash e_2}(v_2)$ has color $1$. Since $v_2$ is a good vertex, one can see that $c(t)=0$. By Lemma \ref{lemma2}, two vertices of $N_{\mathcal{H}_1 \backslash e_1}(v_1)$ have the same color, and since $v_1$ is a good vertex, each vertex of $N_{\mathcal{H}_1}(v_1)$ has color $1$. On the other hand, by Lemma \ref{lemma2}, since $u_2$ is a good vertex, every vertex of $N_{\mathcal{H}_2}(u_2)$ has color $0$.
Now, since $t$ is good, we have $c(x)=1$.  Since $x$ is good, we find that $c(y)=1$ or $c(z)=1$. With no loss of generality assume that $c(y)=1$. Since $u_3$ is good by Lemma \ref{lemma2}, every vertex of $N_{\mathcal{H}_3}(u_3)$ has color $0$. Since $v_4$ is good, by Lemma \ref{lemma2}, we find that $c(u_4)=1$ which contradicts $z$ is good. The proof is complete.
\end{proof}

Now, we have the following corollary.
\begin{corollary}\label{coro2}
There exists a connected cubic graph containing a triangle that has no two vertex disjoint total dominating sets.
\end{corollary}
\begin{proof}
Consider the graph $G$ stated in Theorem \ref{theorem1}. If $G$ has two vertex disjoint total dominating sets $S_0$ and $S_1$, then color all vertices of $S_0$ and $S_1$ by $0$ and $1$, respectively, and color $V(G)\backslash(S_0 \cup S_1)$ by $1$. Obviously, this is a $2$-coupon coloring of $G$ which contradicts Theorem \ref{theorem1}.
\end{proof} 

\begin{theorem}\label{theorem3}
    Let $G$ be a connected cubic graph and $S \subsetneq V(G)$. If $G[S]$ has a $2$-coupon coloring, and $G \backslash S$ has a $\left\{ 1,2 \right\} $-factor, then $G$ has a $2$-coupon coloring.
\end{theorem}
\begin{proof}
First assume that $G=(X,Y)$ is a connected bipartite cubic graph.
By contradiction assume that $S$ is the largest set of vertices of $G$ satisfying the assumption and $G\backslash S$ has a $\left\{ 1,2 \right\} $-factor $F$, and $G$ has no $2$-coupon coloring. With no loss of generality, one may assume that $F$ contains no even cycle because every even cycle has a perfect matching.\\ 
Since $G$ is connected, then there exists an edge $uv$ such that $u \in S$ and $v \in V(G\backslash S)$.
Let $A$ be the set of all edges of $P_{2}$-components of $F$. By a {\it good walk} in $G$, we mean a walk starting at $u$ whose edges are alternatively contained in $E(G)\backslash A$ and $A$ such that its first edge is $uv$ and each edge in $E(G)\backslash A$ appears at most once.\\
Let $W$ be the longest good walk in $G$. Suppose that $W: ue_1x_1e_2y_1e_3x_2e_4y_2 \ldots e_kw$, where $e_i$ is in $E(G)\backslash A$ and $A$ if $i$ is odd and even, respectively. Note that since $G$ is bipartite, we have $ \left\{ x_1, x_2,\ldots \right\} \cap \left\{ y_1, y_2 , \ldots \right\} = \varnothing$.\\
Since $W$ has the longest length among all good walks and for every $x_iy_i \in A$, $d(x_i)=d(y_i)=3$, it implies that the number of edges in $E(G)\backslash A$ incident with $x_i$ is the same as the number of edges in $E(G)\backslash A$ incident with $y_i$. Thus $W$ terminates with an edge in $E(G)\backslash A$. So $w \in S$ because if $w \in V(G \backslash S)$, then $w$ has an edge in $A$ and this contradicts $W$ has the longest length. Color $y_1$ by $1-c(u)$, where $c(u)$ is the color of $u$. If $y_i$ is not colored before, color $y_i$ by $1-c(y_{i-1})$, for $i=1, \ldots, \frac{k-1}{2}$. Let $i$ be an arbitrary index, $v=x_i$ and $e=vz \in A$. Suppose that we meet $e$ for the first time in $W$. Obviously, $z$ is still not colored and after we color $z$, $v$ becomes a good vertex. So every $x_i$ has coupon property. Define $c(x_{\frac{k-1}{2}})=1-c(w)$. Also if $x_j$ is not colored, define $c(x_j)=1-c(x_{j+1})$, for $j=\frac{k-1}{2}-1 , \ldots , 1$ and then with the previous argument, every $y_i$ becomes a good vertex. Note that every vertex in $S \cup V(W)$ has the coupon property and $G \backslash (S\cup V(W))$ has a $\left\{ 1,2 \right\} $-factor which contradicts the assumption.\\
Next, assume that $G$ is a connected non-bipartite cubic graph. Let $V(G)=\left\{ v_1, \ldots , v_n \right\}$. Consider the bipartite cubic graph $G^*$. Since $G$ contains an odd cycle, it is easy to see that $G^*$ is a connected bipartite cubic graph. With no loss of generality, suppose that $S= \left\{ v_1, \ldots , v_r\right\}$.
In the graph $G^*$, define $S’=\left\{ u_1, \ldots , u_r \right\}\cup \left\{ w_1, \ldots , w_r \right\}$. Assume that $c$ is a $2$-coupon coloring for $G[S]$. If we define $c(u_i)=c(w_i)=c(v_i)$, then clearly we obtain a $2$-coupon coloring for $G^*[S’]$. By assumption, $G\backslash V(S)$ has a $\left\{ 1,2 \right\} $-factor $F$. If $v_{i_1},v_{i_2},\ldots ,v_{i_{2t}}$ is an even cycle of $F$, then $u_{i_1}w_{i_2}u_{i_3}w_{i_4}\ldots w_{i_{2t}}$ and $w_{i_1}u_{i_2}w_{i_3}u_{i_4}\ldots u_{i_{2t}}$ are two vertex disjoint even cycles in $G^*$.
If $v_{i_1}v_{i_2}\ldots v_{i_{2t+1}}$ is an odd cycle in $G$, then $u_{i_1}w_{i_2}u_{i_3}w_{i_4}\ldots u_{i_{2t+1}}w_{i_1}$ is an even cycle of order $4t+2$ in $G^*$.  Moreover, if $v_iv_j$ is a $P_{2}$-component of $F$, then $u_iw_j$ and $u_jw_i$ form a  two matching edges in $G^*$. These imply that $H\backslash V(S’)$ has a $\left\{ 1,2 \right\}$-factor.\\
Now, by the previous case, $G^*$ has a $2$-coupon coloring $c$. If we define $c(v_i)=c(u_i)$, for $i=1,\ldots ,n$, then one can see that $c$ is a $2$-coupon coloring for $G$.
\end{proof}

Now, we have the following corollary.
\vspace{0.3cm}
\begin{corollary}\label{coro4}
    If $G$ is a connected cubic graph containing a $4$-cycle, then $G$ has a $2$-coupon coloring.
\end{corollary}
\begin{proof}
    First assume that $G=(X,Y)$ is bipartite. If $C$ is a $4$-cycle of $G$, we claim that $G\backslash V(C)$ has a perfect matching. Let $H=(X’,Y’)=G\backslash V(C)$ and $S  \subseteq X’$. By Marriage Theorem [\ref{3}] it suffices to show that $|N_H(S)|\ge|S|$. Clearly, the number of edges between $S$ and $Y’$ is at least $3|S|-2$. If $|N_H(S)|\le|S|-1$, then there exists a vertex in $Y’$ whose degree is at least $\frac{3|S|-2}{|S|-1}>3$, a contradiction. So  $H$ has a perfect matching and the claim is proved. Since $C$ has a $2$-coupon coloring, Theorem \ref{theorem3} yields that $G$ has a $2$-coupon coloring.\\
    Next, suppose that $G$ is non-bipartite. Consider the connected bipartite cubic graph $G^*$. Since $G$ contains a $4$-cycle, $G^*$ contains two vertex disjoint $4$-cycles say $C_1$ and $C_2$. By the same method as we did before, one can see that $G^*\backslash V(C_1)$ has a $\left\{ 1,2 \right\}$-factor. Since $C_1$ has a $2$-coupon coloring, by Theorem \ref{theorem3}, $G^*$ and so $G$ has a $2$-coupon coloring, as desired.
\end{proof}
\begin{remark}
For every positive integer $r\, (r\ge3)$, not divisible by $4$, there exists a connected cubic graph that contains an induced $r$-cycle and has no $2$-coupon coloring.
\end{remark}
\begin{proof}
We define $H_1,H_2$ and $H_3$ as follows:\\
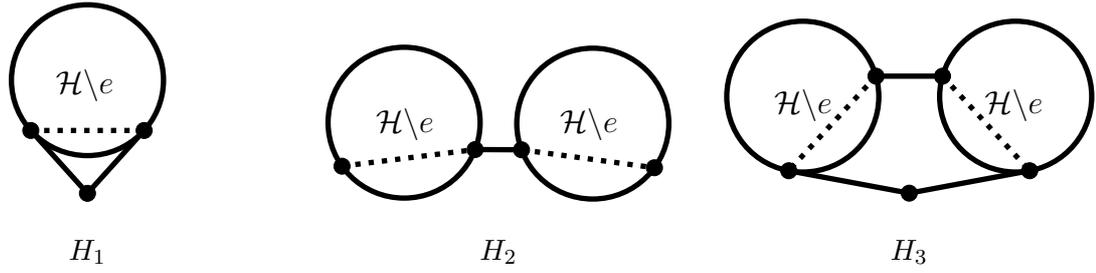
\begin{figure}[htp]
    \setkeys{Gin}{width=\linewidth}
\begin{tblr}{colsep=3pt,
             colspec={@{} *{3}{X[c]}@{}}, 
             measure = vbox}
\centering
\begin{tikzpicture}
\draw [line width=2pt] (0.49,5.36) circle (1cm);
\draw [line width=2pt] (0.49,3.86)-- (1.2353559924999291,4.693333333333332);
\draw [line width=2pt] (0.49,3.86)-- (-0.2553559924999292,4.693333333333332);
\draw (-0.11,5.68) node[anchor=north west,draw=none,fill=none,rectangle,inner sep=5pt] {$\mathcal{H} \backslash e$};
\draw [line width=2pt,dash pattern=on 2pt off 3pt] (-0.2553559924999292,4.693333333333332)-- (1.2353559924999291,4.693333333333332);
\begin{scriptsize}
\draw [fill=black] (0.49,3.86) circle (3pt);
\draw [fill=black] (-0.2553559924999292,4.693333333333332) circle (3pt);
\draw [fill=black] (1.2353559924999291,4.693333333333332) circle (3pt);
\end{scriptsize}
\end{tikzpicture}
\caption*{$H_1$}
    &
\centering
\begin{tikzpicture}
\draw [line width=2pt] (2.00432154600671,4.348177441116346) circle (1cm);
\draw [line width=2pt] (4.486224366859763,4.335727207049599) circle (1cm);
\draw [line width=2pt] (2.9385764369408065,3.9915714978600416)-- (3.5484388236588074,3.9885121800659067);
\draw (1.45,4.72) node[anchor=north west,draw=none,fill=none,rectangle,inner sep=5pt] {$\mathcal{H} \backslash e$};
\draw (3.87,4.72) node[anchor=north west,draw=none,fill=none,rectangle,inner sep=5pt] {$\mathcal{H} \backslash e$};
\draw [line width=2pt,dash pattern=on 2pt off 3pt] (1.1877537518012418,3.7709278958879295)-- (2.9385764369408065,3.9915714978600416);
\draw [line width=2pt,dash pattern=on 2pt off 3pt] (3.5484388236588074,3.9885121800659067)-- (5.2969597740277585,3.7503144470553025);
\begin{scriptsize}
\draw [fill=black] (1.1877537518012418,3.7709278958879295) circle (3pt);
\draw [fill=black] (2.9385764369408065,3.9915714978600416) circle (3pt);
\draw [fill=black] (3.5484388236588074,3.9885121800659067) circle (3pt);
\draw [fill=black] (5.2969597740277585,3.7503144470553025) circle (3pt);
\end{scriptsize}
\end{tikzpicture}
\caption*{$H_2$}
    &
\centering
\begin{tikzpicture}
\draw [line width=2pt] (-0.91,5.32) circle (1cm);
\draw [line width=2pt] (1.89,5.32) circle (1cm);
\draw [line width=2pt] (0.49,4.04)-- (-1.0943327112027714,4.337136097121968);
\draw [line width=2pt] (0.05245632169632408,5.591436601855354)-- (0.9275436783036834,5.591436601855345);
\draw [line width=2pt] (0.49,4.04)-- (2.0743327112027647,4.337136097121975);
\draw (1.3,5.6) node[anchor=north west,draw=none,fill=none,rectangle,inner sep=5pt] {$\mathcal{H} \backslash e$};
\draw (-1.47,5.6) node[anchor=north west,draw=none,fill=none,rectangle,inner sep=5pt] {$\mathcal{H} \backslash e$};
\draw [line width=2pt,dash pattern=on 2pt off 3pt] (0.05245632169632408,5.591436601855354)-- (-1.0943327112027714,4.337136097121968);
\draw [line width=2pt,dash pattern=on 2pt off 3pt] (0.9275436783036834,5.591436601855345)-- (2.0743327112027647,4.337136097121975);
\begin{scriptsize}
\draw [fill=black] (0.49,4.04) circle (3pt);
\draw [fill=black] (0.05245632169632408,5.591436601855354) circle (3pt);
\draw [fill=black] (0.9275436783036834,5.591436601855345) circle (3pt);
\draw [fill=black] (2.0743327112027647,4.337136097121975) circle (3pt);
\draw [fill=black] (-1.0943327112027714,4.337136097121968) circle (3pt);
\end{scriptsize}
\end{tikzpicture}
\caption*{$H_3$}
\end{tblr}
\end{figure}
\\
\newpage
\noindent Now, three cases can be considered.\\
\textbf{Case 1.}
$r=1\,(mod \, 4)$. Consider an $r$-cycle $C$ with the vertex set $\left\{v_1,\ldots,v_r  \right\}$. We join $v_1$ to the $2$-vertex of $H_1$. Now, consider $\frac{r-1}{2}$ disjoint copies of $H_2$.
Join $v_{2k}$ to one $2$-vertex and $v_{2k+1}$ to another $2$-vertex of $k$-th copy of $H_2$, for $k=1,2,\ldots,\frac{r-1}{2}$.\\
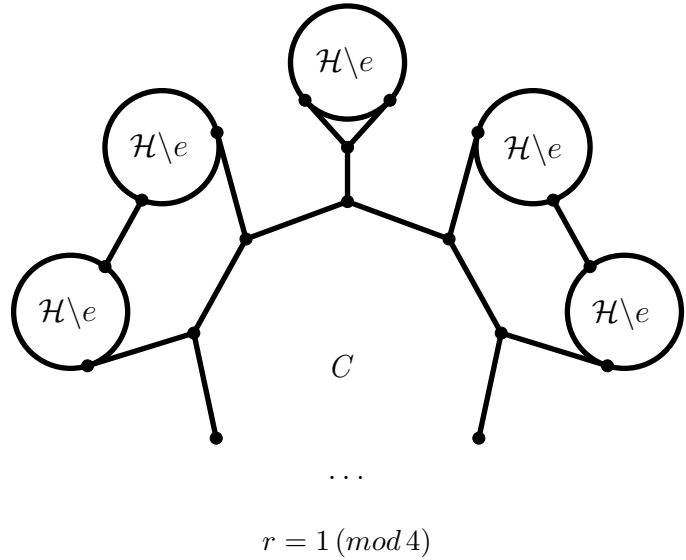
\begin{figure}[h]
\centering
\begin{tikzpicture}[line cap=round,line join=round,>=triangle 45,x=1cm,y=1cm,scale=0.75]
\draw [line width=2pt] (0.49,2.9016374245060907)-- (-1.29,2.24);
\draw [line width=2pt] (0.49,2.9016374245060907)-- (2.27,2.24);
\draw [line width=2pt] (-1.29,2.24)-- (-2.205679855413668,0.5763610716574773);
\draw [line width=2pt] (2.27,2.24)-- (3.185679855413668,0.5763610716574772);
\draw [line width=2pt] (3.185679855413668,0.5763610716574772)-- (2.7924100465635124,-1.2814610239571742);
\draw [line width=2pt] (-2.205679855413668,0.5763610716574773)-- (-1.812410046563512,-1.2814610239571742);
\draw [line width=2pt] (3.743415552046838,3.8633513999727724) circle (1cm);
\draw [line width=2pt] (5.345439942149063,0.9527379101475304) circle (1cm);
\draw [line width=2pt] (-2.7634155520468395,3.863351399972771) circle (1cm);
\draw [line width=2pt] (-4.365439942149063,0.952737910147531) circle (1cm);
\draw [line width=2pt] (0.49,5.36) circle (1cm);
\draw [line width=2pt] (0.49,3.86)-- (1.2353559924999291,4.693333333333332);
\draw [line width=2pt] (0.49,3.86)-- (-0.2553559924999292,4.693333333333332);
\draw [line width=2pt] (4.095806397200763,2.9274984776330966)-- (4.743292673969191,1.7511229472214438);
\draw [line width=2pt] (3.185679855413668,0.5763610716574772)-- (5.048852375814843,-0.0022677523990022467);
\draw [line width=2pt] (2.27,2.24)-- (2.7778964754415543,4.123683712049637);
\draw [line width=2pt] (-1.29,2.24)-- (-1.7978964754415592,4.123683712049635);
\draw [line width=2pt] (-3.115806397200759,2.9274984776331)-- (-3.76329267396919,1.7511229472214462);
\draw [line width=2pt] (-2.205679855413668,0.5763610716574773)-- (-4.068852375814843,-0.0022677523990035976);
\draw [line width=2pt] (0.49,3.86)-- (0.49,2.9016374245060907);
\draw (-0.1,-1.7) node[anchor=north west,draw=none,fill=none,rectangle,inner sep=5pt] {$\mathit{\ldots}$};
\draw (-0.1,0.42) node[anchor=north west,draw=none,fill=none,rectangle,inner sep=5pt] {$\mathit{C}$};
\draw (2.99,4.34) node[anchor=north west,draw=none,fill=none,rectangle,inner sep=5pt] {$\mathcal{H} \backslash e$};
\draw (4.53,1.44) node[anchor=north west,draw=none,fill=none,rectangle,inner sep=5pt] {$\mathcal{H} \backslash e$};
\draw (-0.31,5.88) node[anchor=north west,draw=none,fill=none,rectangle,inner sep=5pt] {$\mathcal{H} \backslash e$};
\draw (-3.57,4.38) node[anchor=north west,draw=none,fill=none,rectangle,inner sep=5pt] {$\mathcal{H} \backslash e$};
\draw (-5.17,1.46) node[anchor=north west,draw=none,fill=none,rectangle,inner sep=5pt] {$\mathcal{H} \backslash e$};
\begin{scriptsize}
\draw [fill=black] (-1.29,2.24) circle (3pt);
\draw [fill=black] (2.27,2.24) circle (3pt);
\draw [fill=black] (0.49,2.9016374245060907) circle (3pt);
\draw [fill=black] (3.185679855413668,0.5763610716574772) circle (3pt);
\draw [fill=black] (-2.205679855413668,0.5763610716574773) circle (3pt);
\draw [fill=black] (2.7924100465635124,-1.2814610239571742) circle (3pt);
\draw [fill=black] (-1.812410046563512,-1.2814610239571742) circle (3pt);
\draw [fill=black] (-1.7978964754415592,4.123683712049635) circle (3pt);
\draw [fill=black] (-3.115806397200759,2.9274984776331) circle (3pt);
\draw [fill=black] (-3.76329267396919,1.7511229472214462) circle (3pt);
\draw [fill=black] (-4.068852375814843,-0.0022677523990035976) circle (3pt);
\draw [fill=black] (2.7778964754415543,4.123683712049637) circle (3pt);
\draw [fill=black] (4.095806397200763,2.9274984776330966) circle (3pt);
\draw [fill=black] (4.743292673969191,1.7511229472214438) circle (3pt);
\draw [fill=black] (5.048852375814843,-0.0022677523990022467) circle (3pt);
\draw [fill=black] (0.49,3.86) circle (3pt);
\draw [fill=black] (-0.2553559924999292,4.693333333333332) circle (3pt);
\draw [fill=black] (1.2353559924999291,4.693333333333332) circle (3pt);
\end{scriptsize}
\end{tikzpicture}
\caption*{$r=1\,(mod \, 4)$}
\end{figure}
\\
\textbf{Case 2.}
$r=1\,(mod \, 4)$. Consider an $r$-cycle $C$ with the vertex set $\left\{v_1,\ldots,v_r  \right\}$. We join $v_1$ to the $2$-vertex of $H_3$. Now, consider $\frac{r-1}{2}$ disjoint copies of $H_2$.
Join $v_{2k}$ to one $2$-vertex and $v_{2k+1}$ to another $2$-vertex of $k$-th copy of $H_2$, for $k=1,2,\ldots,\frac{r-1}{2}$.\\
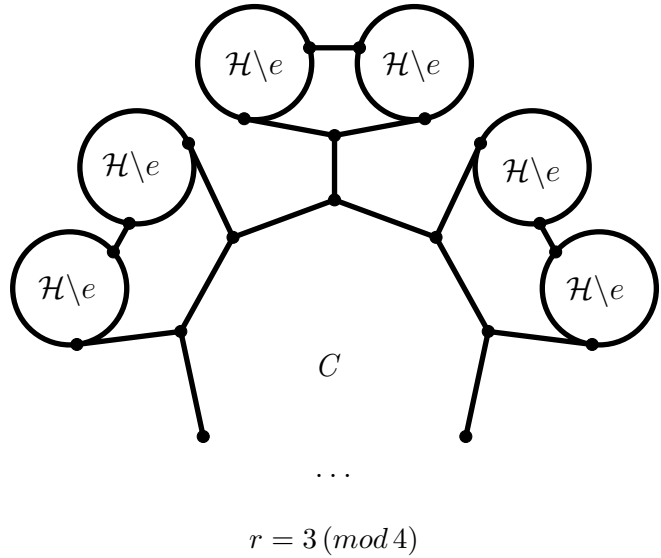
\begin{figure}[htp]
\centering
\begin{tikzpicture}[line cap=round,line join=round,>=triangle 45,x=1cm,y=1cm,scale=0.75]
\draw [line width=2pt] (0.49,2.9016374245060907)-- (-1.29,2.24);
\draw [line width=2pt] (0.49,2.9016374245060907)-- (2.27,2.24);
\draw [line width=2pt] (-1.29,2.24)-- (-2.205679855413668,0.5763610716574773);
\draw [line width=2pt] (2.27,2.24)-- (3.185679855413668,0.5763610716574772);
\draw [line width=2pt] (3.185679855413668,0.5763610716574772)-- (2.7924100465635124,-1.2814610239571742);
\draw [line width=2pt] (-2.205679855413668,0.5763610716574773)-- (-1.812410046563512,-1.2814610239571742);
\draw [line width=2pt] (3.9554835223983495,3.478058954315573) circle (1cm);
\draw [line width=2pt] (5.133371971797551,1.3380303558047293) circle (1cm);
\draw [line width=2pt] (-2.975483522398351,3.478058954315572) circle (1cm);
\draw [line width=2pt] (-4.153371971797551,1.3380303558047302) circle (1cm);
\draw [line width=2pt] (4.090043899881568,2.4871535578162054)-- (4.368177905983735,1.9818300464254397);
\draw [line width=2pt] (3.185679855413668,0.5763610716574772)-- (5.007916586412364,0.3459310934938173);
\draw [line width=2pt] (2.27,2.24)-- (3.050177624305767,3.902819158288093);
\draw [line width=2pt] (-1.29,2.24)-- (-2.070177624305765,3.902819158288097);
\draw [line width=2pt] (-3.1100438998815703,2.4871535578162005)-- (-3.388177905983735,1.9818300464254393);
\draw [line width=2pt] (-2.205679855413668,0.5763610716574773)-- (-4.027916586412363,0.3459310934938191);
\draw [line width=2pt] (0.49,4.04)-- (0.49,2.9016374245060907);
\draw (-0.1,-1.7) node[anchor=north west,draw=none,fill=none,rectangle,inner sep=5pt,draw=none,fill=none,rectangle,inner sep=5pt] {$\mathit{\ldots}$};
\draw (-0.1,0.42) node[anchor=north west,draw=none,fill=none,rectangle,inner sep=5pt] {$\mathit{C}$};
\draw (3.19,3.9) node[anchor=north west,draw=none,fill=none,rectangle,inner sep=5pt] {$\mathcal{H} \backslash e$};
\draw (4.31,1.82) node[anchor=north west,draw=none,fill=none,rectangle,inner sep=5pt] {$\mathcal{H} \backslash e$};
\draw (-3.83,4) node[anchor=north west,draw=none,fill=none,rectangle,inner sep=5pt] {$\mathcal{H} \backslash e$};
\draw (-4.95,1.84) node[anchor=north west,draw=none,fill=none,rectangle,inner sep=5pt] {$\mathcal{H} \backslash e$};
\draw [line width=2pt] (-0.91,5.32) circle (1cm);
\draw [line width=2pt] (1.89,5.32) circle (1cm);
\draw [line width=2pt] (0.49,4.04)-- (-1.0943327112027714,4.337136097121968);
\draw [line width=2pt] (0.05245632169632408,5.591436601855354)-- (0.9275436783036834,5.591436601855345);
\draw [line width=2pt] (0.49,4.04)-- (2.0743327112027647,4.337136097121975);
\draw (1.09,5.8) node[anchor=north west,draw=none,fill=none,rectangle,inner sep=5pt] {$\mathcal{H} \backslash e$};
\draw (-1.67,5.8) node[anchor=north west,draw=none,fill=none,rectangle,inner sep=5pt] {$\mathcal{H} \backslash e$};
\begin{scriptsize}
\draw [fill=black] (-1.29,2.24) circle (3pt);
\draw [fill=black] (2.27,2.24) circle (3pt);
\draw [fill=black] (0.49,2.9016374245060907) circle (3pt);
\draw [fill=black] (3.185679855413668,0.5763610716574772) circle (3pt);
\draw [fill=black] (-2.205679855413668,0.5763610716574773) circle (3pt);
\draw [fill=black] (2.7924100465635124,-1.2814610239571742) circle (3pt);
\draw [fill=black] (-1.812410046563512,-1.2814610239571742) circle (3pt);
\draw [fill=black] (-2.070177624305765,3.902819158288097) circle (3pt);
\draw [fill=black] (-3.1100438998815703,2.4871535578162005) circle (3pt);
\draw [fill=black] (-3.388177905983735,1.9818300464254393) circle (3pt);
\draw [fill=black] (-4.027916586412363,0.3459310934938191) circle (3pt);
\draw [fill=black] (3.050177624305767,3.902819158288093) circle (3pt);
\draw [fill=black] (4.090043899881568,2.4871535578162054) circle (3pt);
\draw [fill=black] (4.368177905983735,1.9818300464254397) circle (3pt);
\draw [fill=black] (5.007916586412364,0.3459310934938173) circle (3pt);
\draw [fill=black] (0.49,4.04) circle (3pt);
\draw [fill=black] (0.05245632169632408,5.591436601855354) circle (3pt);
\draw [fill=black] (0.9275436783036834,5.591436601855345) circle (3pt);
\draw [fill=black] (2.0743327112027647,4.337136097121975) circle (3pt);
\draw [fill=black] (-1.0943327112027714,4.337136097121968) circle (3pt);
\end{scriptsize}
\end{tikzpicture}
\caption*{$r=3\,(mod \, 4)$}
\end{figure}
\\
\newpage
\noindent
\textbf{Case 3.}
$r=2\,(mod\, 4)$. Since $\frac{r}{2}$ is odd, by Cases $1$ and $2$, there exists a connected cubic graph $G$ with an induced $\frac{r}{2}$-cycle which has no $2$-coupon coloring. Now,  $G^*$ satisfies the condition.\\
By a similar argument as we did in the proof of Theorem \ref{theorem1}, one can see that the connected cubic graphs introduced in Cases $1,2$ and $3$ have no $2$-coupon coloring.
\end{proof}
\end{sloppypar}

\end{document}